\documentclass[a4paper,11pt,reqno]{amsart}%,oneside

\pdfoutput=1

\usepackage[ansinew]{inputenc}
\usepackage[english]{babel}
\usepackage{amsmath}
\usepackage{amssymb}
\usepackage{amsthm}
\usepackage{mathrsfs}
\usepackage{verbatim}

\usepackage{mathtools}
\usepackage{ifthen}
\usepackage{tikz}
\usepackage[final,pdfauthor={Sune Precht Reeh, Ergun Yalcin},pdftitle={Representation rings for fusion systems and dimension functions}]{hyperref}
\usepackage{bookmark}

\usepackage[abbrev,msc-links]{amsrefs}

\numberwithin{equation}{section}
\numberwithin{figure}{section}
\numberwithin{table}{section}
\allowdisplaybreaks

\theoremstyle{plain}
\newtheorem{theorem}{Theorem}[section]
\newtheorem{prop}[theorem]{Proposition}
\newtheorem{lemma}[theorem]{Lemma}

\newtheorem{question}[theorem]{Question}
\newtheorem{fact}[theorem]{Fact}

\newtheorem{thmIntro}{Theorem}

\theoremstyle{definition}
\newtheorem{definition}[theorem]{Definition}

\newtheorem{remark}[theorem]{Remark}

\newtheorem{example}[theorem]{Example}

\DeclareMathOperator{\Aut}{Aut}
\DeclareMathOperator{\Gal}{Gal}

\DeclareMathOperator{\Hom}{Hom}

\DeclareMathOperator{\Inj}{Inj}

\newcommand{\Qd}{\mathrm{Qd}}
\newcommand{\rk}{\mathrm{rk}}

\newcommand{\Mor}{\mathrm{Mor}}
\newcommand{\Dim}{\mathrm{Dim}}

\newcommand{\e}{\varepsilon}
\newcommand{\ph}{\varphi}

\newcommand{\x}{\times}
\newcommand{\ox}{\otimes}
\renewcommand{\tilde}{\widetilde}

\renewcommand{\bar}{\overline}

\newcommand{\lc}[1]{\prescript{#1\!}{}}

\newcommand{\cl}{c}

\newcommand{\xto}{\xrightarrow}

\newcommand{\into}{\hookrightarrow}

\def\cA{\mathcal A}\def\cC{\mathcal C}\def\cD{\mathcal D}
\def\cF{\mathcal F}\def\cH{\mathcal H}

\def\cP{\mathcal P}

\def\CC{\mathbb C}
\def\FF{\mathbb F}
\def\KK{\mathbb K}\def\LL{\mathbb L}

\def\QQ{\mathbb Q}\def\RR{\mathbb R}

\def\ZZ{\mathbb Z}

\DeclareMathOperator{\res}{res}
\DeclareMathOperator{\tr}{tr}
\DeclareMathOperator{\ind}{ind}
\DeclareMathOperator{\id}{id}
\DeclareMathOperator{\incl}{incl}
\newcommand{\bifree}{{\mathord{\text{bifree}}}}

\DeclarePairedDelimiter{\abs}\lvert\rvert
\DeclarePairedDelimiter{\gen}\langle\rangle

\renewcommand{\theenumi}{$(\roman{enumi})$}\renewcommand{\labelenumi}{\theenumi}

\usetikzlibrary{arrows,matrix,decorations,positioning,calc}
\tikzset{dot/.style={circle,fill=black,thick,inner sep=0pt,minimum size=1mm,draw}}
\tikzset{arrow/.style={semithick,>=stealth',shorten >=1pt,shorten <=1pt}}
\tikzset{equal/.style={arrow,double distance=2pt}}

\title{Representation rings for fusion systems and dimension functions}
\author[S. P. Reeh]{Sune Precht Reeh}
\address{Department of Mathematics,  Massachusetts Institute of Technology, Cambridge, MA, USA}
\email{reeh@mit.edu}
\author[E. Yal\c{c}{\i}n]{Erg\"un Yal\c{c}{\i}n}
\address{Department of Mathematics, Bilkent University,
 06800 Bilkent, Ankara, Turkey}
\email{yalcine@fen.bilkent.edu.tr }
\subjclass[2010]{}
\thanks{The first author is supported by the Danish Council for Independent Research's Sapere Aude program (DFF--4002-00224). The second author is  partially supported by the Scientific and Technological Research Council of Turkey (T\" UB\. ITAK)
through the research program B\. IDEB-2219.}

\begin{document}

\begin{abstract} We define the representation ring of a saturated fusion system $\cF$ as the Grothendieck ring of the semiring of $\cF$-stable representations, and study the dimension functions of $\cF$-stable representations using the transfer map induced by the characteristic idempotent of $\cF$. We find a list of conditions for an $\cF$-stable super class function to be realized as the dimension function of an $\cF$-stable virtual representation. We also give an application of our results to constructions of finite group actions on homotopy spheres.
\end{abstract}

\maketitle

\section{Introduction}

Let $G$ be a finite group and $\KK$ be a subfield of the complex numbers. The representation ring $R_{\KK} (G)$  is defined as the Grothendieck ring  of the semiring of the isomorphism classes of finite dimensional $G$-representations over $\KK$. The addition is given by direct sum and the multiplication is given by tensor product over $\KK$. The elements of the representation ring can be taken as virtual representations $U-V$ up to isomorphism.

The dimension function associated to a $G$-representation $V$ over $\KK$ is a function
\[\Dim V\colon \mathrm{Sub} (G) \to \ZZ,\]
from subgroups of $G$ to integers, defined by $(\Dim V) (H)=\dim _{\KK } V^H$ for every $H \leq G$. Extending the dimension function linearly, we obtain a group homomorphism
\[\Dim\colon R_{\KK} (G)\to C(G),\]
where $C(G)$ denotes the group of super class functions, i.e. functions $f\colon\mathrm{Sub} (G) \to \ZZ$ that are constant on conjugacy classes.
When $\KK=\RR$, the real numbers, and $G$ is nilpotent, the image of the $\Dim$ homomorphism is equal to the group of Borel-Smith functions $C_b (G)$. This is the subgroup of $C(G)$ formed by super class functions satisfying certain conditions known as Borel-Smith conditions (see Definition \ref{def:BorelSmith} and Theorem \ref{thm:nilpotent}).

If we restrict a $G$-representation to a Sylow $p$-subgroup $S$ of $G$, we obtain an $S$-representation which respects fusion in $G$, i.e. character values for $G$-conjugate elements of $S$ are equal. Similarly the restriction of a Borel-Smith function $f \in C_b(G)$ to a Sylow $p$-group $S$ gives a Borel-Smith function $f \in C_b (S)$ which is constant on $G$-conjugacy classes of subgroups of $S$. The question we would like to answer is the following:

\begin{question}\label{ques:main} Given a Borel-Smith function $f \in C_b (S)$ that is constant on $G$-conjugacy classes of subgroups of $S$, under what conditions can we find a real $S$-representation $V$ such that $V$ respects fusion in $G$ and $\Dim V=f$?
\end{question}

To study this problem we introduce representation rings for abstract fusion systems and study the dimension homomorphism for fusion systems.  Let $\cF$ be a (saturated) fusion system on a finite $p$-group $S$ (see Section \ref{sectBisetFusion} for a definition). We define the representation ring $R_{\KK} (\cF)$ as the subring of $R_{\KK } (S)$ formed by virtual representations that are $\cF$-stable. A (virtual) $S$-representation $V$ is said to be $\cF$-stable if for every morphism $\varphi\colon P \to S$ in the fusion system $\cF$, we have
\[\res ^S _P V= \res _{\varphi}V\]
where $\res _{\varphi} V$ is the $P$-representation with the left $P$-action given by $p\cdot v=\varphi (p)v$ for every $v \in V$.

We prove that $R_{\KK} (\cF)$ is equal to the Grothendieck ring of the semiring of $\cF$-stable $S$-representations over $\KK$ (Proposition \ref{propVirtualFstableReps}). This allows us to define the dimension homomorphism
\[\Dim\colon R_{\KK } (\cF) \to C(\cF)\]
as the restriction of the dimension homomorphism for $S$. Here $C(\cF)$ denotes the group of super class functions $f\colon \mathrm{Sub(S)} \to \ZZ$ that are constant on $\cF$-conjugacy classes of subgroups of $S$. In a similar way, we define $C_b(\cF)$ as the group of Borel-Smith functions which are constant on $\cF$-conjugacy classes of subgroups of $S$.

Our first observation is that after localizing at $p$, the dimension homomorphism for fusion systems
\[\Dim\colon R_{\RR} (\cF) _{(p)}\to C_b (\cF)_{(p)}\]
is surjective (see Proposition \ref{p-localSurj}). This follows from a more general result on short exact sequences of biset functors (see Proposition \ref{propShortExactBisetFunctors}).

To obtain a surjectivity result for the dimension homomorphism in integer coefficients, we consider a result due to S. Bauer \cite{Bauer} for realization of Borel-Smith functions defined on subgroups of prime power order in a finite group $G$. In his paper Bauer introduces a new condition in addition to the Borel-Smith conditions. We introduce a similar condition for fusion systems (see Definition \ref{def:FusionCondition}). The group of Borel-Smith functions for $\cF$ which satisfy this extra condition is denoted by $C_{ba} (\cF)$. We prove the following theorem which is one of the main results of this paper.

\begin{thmIntro}\label{thm:IntroMain} Let $\cF$ be a saturated fusion system on a $p$-group $S$. Then
\[\Dim \colon R_{\RR} (\cF ) \to C_{ba} (\cF)\]
is surjective.
\end{thmIntro}

We prove this theorem in Section \ref{sect:Bauer} as Theorem \ref{thm:main}.
If instead of virtual representations we wish to find an actual representation realizing a given $\cF$-stable super class function, we first observe that such a function must be \emph{monotone}, meaning that for every $K\leq H \leq S$, we must have $f(K) \geq f(H) \geq 0$. It is an interesting question if every monotone super class function $f \in C_{ba} (\cF)$ is realized as the dimension function of an actual $\cF$-stable $S$-representation. We answer this question up to multiplication with a positive integer.

\begin{thmIntro}\label{thm:IntroMonotone} Let $\cF$ be a saturated fusion system on a $p$-group $S$. For every monotone Borel-Smith function $f \in C_b(\cF)$, there exists an integer $N \geq 1$ and an $\cF$-stable rational $S$-representation $V$ such that $\Dim V=Nf$.
\end{thmIntro}

This theorem is proved as Theorem \ref{thm:monotone} in the paper. From the proof of this theorem we observe that it is possible to choose the integer $N$ independent from the super class function $f$. We also note that it is enough that the function $f$ only satisfies the condition \ref{def:BorelSmith2} of the Borel-Smith conditions given in Definition \ref{def:BorelSmith} for the conclusion of the theorem to hold.

In the rest of the paper we give some applications of our results to constructing finite group actions on finite homotopy spheres. Note that if $X$ is a finite-dimensional $G$-CW-complex which is homotopy equivalent to a sphere, then by Smith theory, for each $p$-group $H \leq G$, the fixed point subspace $X^H$ has mod-$p$ homology of a sphere $S^{\underline{n} (H)}$. We define the dimension function of $X$ to be the super class function $\Dim_{\cP} X \colon \cP \to \ZZ$ such that $(\Dim _{\cP} X) (H)=\underline{n} (H)+1$ for every $p$-subgroup $H \leq G$, over all primes dividing the order of $G$. We prove the following theorem.

\begin{thmIntro}\label{thm:IntromainApp} Let $G$ be a finite group, and let $f\colon \cP \to \ZZ$ be a monotone Borel-Smith function. Then there is an integer $N\geq 1$ and a finite $G$-CW-complex $X \simeq S^n$, with prime power isotropy, such that $\Dim_{\cP} X=Nf$.
\end{thmIntro}

This theorem is proved as Theorem \ref{thm:mainApp} in the paper. Up to multiplying with an integer, Theorem \ref{thm:IntromainApp} answers the question of when a monotone Borel-Smith function defined on subgroups of prime power order can be realized. We believe this is a useful result for constructing finite homotopy $G$-spheres with certain restrictions on isotropy subgroups (for example rank restrictions). We also prove a similar result (Theorem \ref{thm:AlgMainApp}) on algebraic homotopy $G$-spheres providing a partial answer to a question of Grodal and Smith \cite{GrodalAbs}.

The paper is organized as follows: In Section \ref{sectBisetFusion} we introduce basic definitions of biset functors and fusion systems, and prove Proposition \ref{propShortExactBisetFunctors} which is one of the key results for the rest of the paper. In Section \ref{sectRepresentations}, we introduce representation rings and dimension homomorphisms for fusion systems. In Section \ref{sect:BorelSmith}, we define Borel-Smith functions and prove Proposition \ref{p-localSurj} which states that the dimension homomorphism is surjective in $p$-local coefficients. In Section \ref{sect:Bauer}, we consider a theorem of Bauer for finite groups and prove Theorem A using Bauer's theorem and its proof.
Theorem B is proved in Section \ref{sect:monotone} using some results on rational representation rings for finite groups. In the last section, Section \ref{sect:applications}, we give some applications of our results to constructions of finite group actions on homotopy spheres. Throughout the paper we assume all the fusion systems are saturated unless otherwise is stated clearly.

\subsection*{Acknowledgement}  We thank Matthew Gelvin, Jesper Grodal, and Bob Oliver for many helpful comments on the first version of the paper. Most of the work in this paper was carried out in May-June 2015 when the authors were visiting McMaster University. Both authors thank Ian Hambleton for hosting them at McMaster University. The second author would like to thank McMaster University for the support provided by a H.~L.~Hooker Visiting Fellowship.

\section{Biset functors and fusion systems}
\label{sectBisetFusion}

Let $P$ and $Q$ be finite groups. A (finite) $(P,Q)$-set is a finite set $X$ equipped with a left action of $P$ and a right action of $Q$, and such that the two actions commute. The isomorphism classes of $(P,Q)$-bisets form a monoid with disjoint union as addition, and we denote the group completion of this monoid by $A(P,Q)$. The group $A(P,Q)$ is \emph{the Burnside biset module} for $P$ and $Q$, and it consists of \emph{virtual $(P,Q)$-bisets} $X-Y$ where $X,Y$ are actual $(P,Q)$-bisets (up to isomorphism).

The biset module $A(P,Q)$ is a free abelian group generated by the transitive $(P,Q)$-bisets $(P\x Q)/D$ where the subgroup $D\leq P\x Q$ is determined up to conjugation in $P\x Q$. A special class of transitive $(P,Q)$-bisets are the ones where the right $Q$-action is free. These $Q$-free basis elements are determined by a subgroup $R\leq P$ and a group homomorphism $\ph\colon R\to Q$, and the corresponding transitive biset is denoted
\[[R,\ph]_P^Q := (P\x Q)/(pr,q)\sim(p,\ph(r)q) \text{ for $p\in P$, $q\in Q$ and $r\in R$.}\]
We write $[R,\ph]$ whenever the groups $Q,P$ are clear from the context.

The biset modules form a category $\cA$ whose objects are finite groups, and the morphisms are given by $\Hom _{\cA} (P, Q)=A(P,Q)$ with the associative composition $\circ\colon A(Q,R)\x A(P,Q) \to A(P,R)$ defined by
\[Y\circ X := X \x_Q Y\in A(P,R)\]
for bisets $X\in A(P,Q)$ and $Y\in A(Q,R)$, and then extended bilinearly to all virtual bisets. Note that this category is the opposite category of \emph{the biset category} for finite groups defined in \cite{Bouc}*{Definition 3.1.1}. Given any commutative   ring $R$ with identity, we also have an $R$-linear category $R\cA$ with morphism sets $\Mor_{R\cA}(G,H) = R\ox A(G,H)$.

\begin{definition}
Let $\cC$ be a collection of finite groups closed under taking subgroups and quotients, and let $R$ be a commutative ring with identity ($R=\ZZ$ unless specified otherwise).
A \emph{biset functor} on $\cC$ and over $R$ is an $R$-linear \emph{contravariant} functor from $R\cA_{\cC}$  to $R\text{-mod}$. Here $R\cA_\cC$ is the full subcategory of the biset category $R\cA$ restricted to groups in $\cC$, i.e. the set of morphisms between $P,Q\in \cC$ is the Burnside biset module $R\ox A(P,Q)$. In particular a biset functor has restrictions, inductions, isomorphisms, inflations and deflations between the groups in $\cC$.

A \emph{global Mackey functor} on $\cC$ and over $R$ is an $R$-linear contravariant functor from $R\cA^\bifree_\cC$ to $R\text{-mod}$, where $R\cA^\bifree_\cC$ has just the \emph{bifree} (virtual) bisets for groups in $\cC$, i.e. for the morphism sets $R\ox A^{\bifree}(P,Q)$ with $P,Q\in \cC$. A global Mackey functor has restrictions, inductions and isomorphisms between the groups in $\cC$.

For $Q\leq P$ in $\cC$ and a given biset functor/global Mackey functor $M$, we use the following notation for restriction and induction maps:
\[\res_Q^P\colon M(P)\xto{[Q,\incl]_Q^P} M(Q), \qquad  \ind_Q^P \colon M(Q)\xto{[Q,\id]_P^Q} M(P).\]
Similarly for a homomorphism $\ph\colon R\to P$, we denote the restriction along $\ph$ by
\[\res_\ph\colon M(P)\xto{[R,\ph]_R^P} M(R).\]

By a \emph{biset functor/global Mackey functor defined on a $p$-group $S$}, we mean a biset functor/global Mackey functor on some collection $\cC$ containing $S$.
\end{definition}

A fusion system is an algebraic structure that emulates the $p$-structure of a finite group: We take a finite $p$-group $S$ to play the role of a Sylow $p$-subgroup and endow $S$ with additional conjugation structure. To be precise \emph{a fusion system on $S$} is a category $\cF$ where the objects are the subgroups $P\leq S$ and the morphisms $\cF(P,Q)$ satisfy two properties:
\begin{itemize}
\item For all $P,Q\leq S$ we have $\Hom_S(P,Q)\subseteq \cF(P,Q) \subseteq \Inj(P,Q)$, where $\Inj(P,Q)$ are the injective group homomorphisms and $\Hom_S(P,Q)$ are all homomorphisms $P\to Q$ induced by conjugation with elements of $S$.
\item Every morphism $\ph\in \cF(P,Q)$ can be factored as $P\xto{\ph} \ph P\into Q$ in $\cF$, and the inverse homomorphism $\ph^{-1}\colon \ph P\to P$ is also in $\cF$.
\end{itemize}
We think of the morphisms in $\cF$ as conjugation maps to the point that we say that subgroups of $S$ or elements in $S$ are \emph{$\cF$-conjugate} whenever they are connected by an isomorphism in $\cF$.

There are many examples of fusion systems: Whenever $S$ is a $p$-subgroup in an ambient group $G$, we get a fusion system $\cF_S(G)$ by defining $\Hom_{\cF_S(G)}(P,Q) := \Hom_G(P,Q)$ as the set of all homomorphisms $P\to Q$ induced by conjugation with elements of $G$.
To capture the ``nice'' fusion systems that emulate the case when $S$ is Sylow in $G$, we need two additional axioms. These axioms require a couple of additional notions: We say that a subgroup $P\leq S$ is \emph{fully $\cF$-centralized} if $\abs{C_S(P)}\geq \abs{C_S(P')}$ whenever $P$ and $P'$ are conjugate in $\cF$. Similarly, $P\leq S$ is \emph{fully $\cF$-normalized} if $\abs{N_S(P)}\geq \abs{N_S(P')}$ when $P,P'$ are conjugate in $\cF$.

We say that a fusion system $\cF$ on a $p$-group $S$ is \emph{saturated} if it has the following properties:
\begin{itemize}
\item If $P\leq S$ is fully $\cF$-normalized, then $P$ is also fully $\cF$-centralized and $\Aut_S(P)$ is a Sylow $p$-subgroup of $\Aut_\cF(P)$.
\item If $\ph\in \cF(P,S)$ is such that the image $\ph P$ is fully $\cF$-centralized, then $\ph$ extends to a map $\tilde \ph\in \cF(N_\ph,S)$ where
    \[N_\ph := \{x\in N_S(P) \mid \exists y\in N_S(\ph P)\colon \ph\circ c_x = c_y\circ \ph \in \cF(P,S)\}.\]
\end{itemize}
The main examples of saturated fusion systems are $\cF_S(G)$ whenever $S$ is a Sylow $p$-subgroup of $G$, however, other \emph{exotic} saturated fusion systems exist.

For a bisets functor or global Mackey functor $M$ that is defined on $S$, we can evaluate $M$ on any fusion system $\cF$ over $S$ according to the following definition:
\begin{definition}\label{defFstable}
Let $\cF$ be a fusion system on $S$, and let $M$ be a biset functor/global Mackey functor defined on $S$. We define $M(\cF)$ to be the set of $\cF$-stable elements in $M(S)$:
\[M(\cF) := \{X\in M(S)\mid \res_\ph X = \res_P X \text{ for all $P\leq S$ and $\ph\in \cF(P,S)$}\}.\]
\end{definition}

There is a close relationship between saturated fusion systems on $S$ and certain bisets in the double Burnside ring $A(S,S)$:
\begin{definition}\label{defFcharacteristic}
A virtual $(S,S)$-biset $X\in A(S,S)$ (or more generally in $A(S,S)_{(p)}$ or $A(S,S)^\wedge_p$) is said to be \emph{$\cF$-characteristic} if
\begin{itemize}
  \item $X$ is a linear combination of the transitive bisets on the form $[P,\ph]_S^S$ with $P\leq S$ and $\ph\in \cF(P,S)$.
  \item $X$ is $\cF$-stable for the left and the right actions of $S$. If we consider $A(S,S)$ as a biset functor in both variables, the requirement is that $X\in A(\cF,\cF)$, cf. Definition \ref{defFstable}.
  \item $\abs{X}/\abs{S}$ is not divisible by $p$.
\end{itemize}
\end{definition}

There are a couple of particular $\cF$-characteristic elements $\Omega_\cF$ and $\omega_\cF$ that we shall make use of, and that are reasonably well behaved: $\Omega_\cF$ is an actual biset, and $\omega_\cF$ is idempotent. These particular elements have been studied in a series of papers \citelist{\cite{Ragnarsson} \cite{RagnarssonStancu} \cite{ReehIdempotent} \cite{GelvinReeh}}, and we gather their defining properties in the following proposition:
\begin{prop}[\citelist{\cite{Ragnarsson} \cite{RagnarssonStancu} \cite{GelvinReeh}}]
A fusion system $\cF$ on $S$ has characteristic elements in $A(S,S)^\wedge_p$ only if $\cF$ is saturated (see \cite{RagnarssonStancu}*{Corollary 6.7}).
Let $\cF$ be a saturated fusion system on $S$.

$\cF$ has a unique minimal $\cF$-characteristic actual biset $\Omega_\cF$ which is contained in all other $\cF$-characteristic actual bisets (see \cite{GelvinReeh}*{Theorem A}).

Inside $A(S,S)_{(p)}$ there is an idempotent $\omega_\cF$ that is also $\cF$-characteristic. This $\cF$-characteristic idempotent is unique in $A(S,S)_{(p)}$ and even $A(S,S)^\wedge_p$ (see \cite{Ragnarsson}*{Proposition 5.6}, or \cite{ReehIdempotent}*{Theorem B} for a more explicit construction).

It is possible to reconstruct the saturated fusion system $\cF$ from any $\cF$-characteristic element, in particular from $\Omega_\cF$ or $\omega_\cF$ (see \cite{RagnarssonStancu}*{Proposition 6.5}).
\end{prop}

Using the characteristic idempotent $\omega_\cF$ it becomes rather straightforward to calculate the $\cF$-stable elements for any biset functor/global Mackey functor after $p$-localization.

\begin{prop}\label{propIdempotentBisetTransfer}
Let $\cF$ be a saturated fusion system on $S$, and let $M$ be a biset functor/global Mackey functor defined on $S$. We then obtain
\[M(\cF)_{(p)} = \omega_\cF\cdot M(S)_{(p)},\]
and writing $\tr^\cF_S\colon M(S) \xto{\omega_\cF} M(\cF)$, $\res_S^\cF\colon M(\cF)\into M(S)$, we have $\tr_S^\cF \circ \res^\cF_S = \id$.
\end{prop}

\begin{proof}
The inclusion $\omega_\cF\cdot M(S)_{(p)}\leq M(\cF)_{(p)}$ follows immediately from the fact that $\omega_\cF$ is $\cF$-stable: For each $X\in M(S)_{(p)}$, $P\leq S$ and $\ph\in \cF(P,S)$, we get
\[\res_\ph(\omega_\cF \cdot X) = (\omega_\cF\circ [P,\ph]_P^S)\cdot X = (\omega_\cF\circ [P,\incl]_P^S)\cdot X =\res_P(\omega_\cF\cdot X);\]
hence $\omega_\cF \cdot X\in M(\cF)_{(p)}$.

It remains to show that $\omega_\cF \cdot X = X$ for all $X\in M(\cF)_{(p)}$. The following is essentially the proof by Kári Ragnarsson of a similar result \cite{Ragnarsson}*{Corollary 6.4} for maps of spectra.

We write $\omega_\cF$ as a linear combination of $(S,S)$-biset orbits:
\[\omega_\cF = \sum_{(Q,\psi)_{S,S}} c_{Q,\psi} [Q,\psi]_S^S.\]
Suppose then that $X\in M(\cF)_{(p)}$, and we calculate $\omega_\cF\cdot X$ using the decomposition of $\omega_\cF$ above:
\begin{align*}
\omega_\cF\cdot X &= \left(\sum_{(Q,\psi)_{S,S}} c_{Q,\psi} [Q,\psi]_S^S\right)\cdot X
\\ &= \sum_{(Q,\psi)_{S,S}} c_{Q,\psi} \ind_Q^S\res_\psi X
\\ &= \sum_{(Q,\psi)_{S,S}} c_{Q,\psi} \ind_Q^S\res_Q^S X
\\ &= \sum_{(Q)_S} \left(\sum_{(\psi)_{N_S(Q),S}} c_{Q,\psi}\right) \ind_Q^S\res_Q^S X
\end{align*}
According to \cite{Ragnarsson}*{Lemma 5.5} and \cite{ReehIdempotent}*{4.7 and 5.10}, the sum of coefficients $\sum_{(\psi)_{N_S(Q),S}} c_{Q,\psi}$ equals $1$ for $Q=S$ and $0$ otherwise, hence we conclude
\begin{align*}
\omega_\cF\cdot  X &= \sum_{(Q)_S} \left(\sum_{(\psi)_{N_S(Q),S}} c_{Q,\psi}\right) \ind_Q^S\res_Q^S X = \ind_S^S \res_S^S X = X.\qedhere
\end{align*}
\end{proof}

\begin{prop}\label{propShortExactBisetFunctors}
Let $\cF$ be a saturated fusion system on $S$. If
\[0\to M_1\to M_2\to M_3\to 0\]
is a short-exact sequence of biset functors/global Mackey functors defined on $S$, then the induced sequence
\[0\to M_1(\cF)_{(p)} \to M_2(\cF)_{(p)}\to M_3(\cF)_{(p)} \to 0\]
is exact.
\end{prop}

\begin{proof}
Note that exactness of $0\to M_1\to M_2\to M_3\to 0$ implies that
\[0\to (M_1)_{(p)}\to (M_2)_{(p)}\to (M_3)_{(p)}\to 0\]
is exact as well. Additionally we have a sequence of inclusions
\[
\begin{tikzpicture}
\node (M) [matrix of math nodes] {
   0 &[1cm] M_1(\cF)_{(p)} &[1cm] M_2(\cF)_{(p)} &[1cm] M_3(\cF)_{(p)} &[1cm] 0 \\[1cm]
   0 & M_1(S)_{(p)} & M_2(S)_{(p)} & M_3(S)_{(p)} & 0 \\
};
\path[->,arrow]
    (M-1-1) edge (M-1-2)
    (M-1-2) edge (M-1-3)
    (M-1-3) edge (M-1-4)
    (M-1-4) edge (M-1-5)
    (M-1-2) edge[right hook->] (M-2-2)
    (M-1-3) edge[right hook->] (M-2-3)
    (M-1-4) edge[right hook->] (M-2-4)
    (M-2-1) edge (M-2-2)
    (M-2-2) edge (M-2-3)
    (M-2-3) edge (M-2-4)
    (M-2-4) edge (M-2-5)
;
\end{tikzpicture}
\]

We start with proving injectivity and suppose that $X\in M_1(\cF)_{(p)}$ maps to $0$ in $M_2(\cF)_{(p)}$. Considered as an element of $M_1(S)_{(p)}$, then $X$ also maps to $0$ in $M_2(S)_{(p)}$, hence by injectivity of $M_1(S)_{(p)} \to M_2(S)_{(p)}$, $X=0$.

Next up is surjectivity, so we consider an arbitrary $X\in M_3(\cF)_{(p)}$. By surjectivity of $M_2(S)_{(p)} \to M_3(S)_{(p)}$ we can find $Y\in M_2(S)_{(p)}$ that maps to $X$. Then $\omega_\cF\cdot  Y\in M_2(\cF)_{(p)}$ maps to $\omega_\cF\cdot  X$, and since $X$ is $\cF$-stable, $\omega_\cF\cdot  X$ equals $X$.

Finally we consider the middle of the sequence and suppose that $X\in  M_2(\cF)_{(p)}$ maps to $0$ in $M_3(\cF)_{(p)}$. Then there exists a $Y\in M_1(S)_{(p)}$ that maps to $X$. Similarly to above, $\omega_\cF\cdot  Y\in M_1(\cF)_{(p)}$ then maps to $\omega_\cF\cdot  X=X\in M_2(\cF)_{(p)}$.
\end{proof}

\section{Representation rings}
\label{sectRepresentations}

Throughout this entire section we use $\KK$ to denote a subfield of the complex numbers. In applications further on, $\KK$ will be one of $\QQ$, $\RR$ and $\CC$.

\begin{definition}
\label{defRepRing}
Let $S$ be a finite $p$-group. The \emph{representation semiring}  $R^+_\KK(S)$ consists of the isomorphism classes of finite dimensional representations of $S$ over $\KK$; addition is given by direct sum $\oplus$ and multiplication by tensor product $\ox_\KK$.

The \emph{representation ring} $R_\KK(S)$ is the Grothendieck ring of the semiring $R^+_\KK(S)$, and consists of virtual representations $X=U-V\in R_\KK(S)$ with $U,V\in R^+_\KK(S)$.
\end{definition}

Additively $R_\KK^+(S)$ is a free abelian monoid, so it has the cancellation property and $R_\KK^+(S)\to R_\KK(S)$ in injective. In particular for all representations $U_1,U_2,V_1,V_2\in R^+_\KK(S)$ we have $U_1-V_1=U_2-V_2$ in the representation ring if and only if $U_1\oplus V_2$ and $U_2\oplus V_1$ are isomorphic representations.

\begin{remark}
The representation ring $R_\KK(-)$ has the structure of a biset functor over all finite groups. For a $(G,H)$-biset $X$ and an $H$-representation $V$, the product $\KK[X]\otimes_{\KK[H]} V$ is a $\KK$-vector space on which $G$ acts linearly, and we define
\[X^*(V) := \KK[X]\otimes_{\KK[H]} V\in R_\KK(G),\] which extends linearly to all virtual representations $V\in R_\KK(H)$ and virtual bisets $X\in A(G,H)$.

If $K\leq H$ is a subgroup, then $H$ as a $(K,H)$-biset gives restriction $\res^H_K$ of representations from $H$ to $K$, while $H$ as a $(H,K)$-biset gives induction $\ind^H_K$ of representations from $K$ to $H$.
\end{remark}

\begin{definition}
Let $\cF$ be a saturated fusion system on $S$. As described in Definition \ref{defFstable}, we define the \emph{representation ring} $R_\KK(\cF)$ as the ring of $\cF$-stable virtual $S$-representations, i.e.  virtual representations $V$ satisfying $\res_\ph V \cong \res_P V$ for all $P\leq S$ and homomorphisms $\ph\in \cF(P,S)$.
To see that $R_\KK(\cF)$ is a subring and closed under multiplication, note that $\res_\ph$ respects tensor products $\ox_\KK$, so tensor products of $\cF$-stable representations are again $\cF$-stable.

As the restriction of any actual representation is a representation again, the definition above makes sense in $R_\KK^+(S)$ as well, and we define the \emph{representation semiring} $R_\KK^+(\cF)$ to consist of all the $\cF$-stable representations of $S$.
\end{definition}

\begin{prop}\label{propVirtualFstableReps}
Let $\cF$ be a saturation fusion system on $S$. The representation ring $R_\KK(\cF)$ for $\cF$ is the Grothendieck ring of the representation semiring $R_\KK^+(\cF)$.
\end{prop}

\begin{proof}
The inclusion of semirings $R_\KK^+(\cF)\into R_\KK^+(S)$ induces a homomorphism between the Grothendieck rings $f\colon Gr(R_\KK^+(\cF)) \to R_\KK(S)$. The image of this map is the subgroup $\gen{R_\KK^+(\cF)}\leq R_\KK(S)$ generated by $R_\KK^+(\cF)$.
Because $R_\KK(S)$ has cancellation, the map $f\colon Gr(R_\KK^+(\cF)) \to \gen{R_\KK^+(\cF)}\leq R_\KK(S)$ is in fact injective, so $\gen{R_\KK^+(\cF)}\cong Gr(R_\KK^+(\cF))$.

Since all $\cF$-stable representations are in particular $\cF$-stable virtual representations, the inclusion $\gen{R_\KK^+(\cF)}\leq R_\KK(\cF)$ is immediate.
It remains to show that $R_\KK(\cF)$ is contained in $\gen{R_\KK^+(\cF)}$.

Let $X\in R_\KK(\cF)$ be an arbitrary $\cF$-stable virtual representation, and suppose that $X=U-V$ where $U,V$ are actual $S$-representations that are not necessarily $\cF$-stable. Let $\Omega_\cF$ be the minimal characteristic biset for $\cF$, and note that $\Omega_\cF$ has an orbit of the type $[S,id]$ (see \cite{GelvinReeh}*{Theorem 5.3}), so $\Omega_\cF-[S,id]$ is an actual biset.
We can then write $X$ as
\begin{multline}\label{eqDifferenceOfFstableReps}
X=U-V=(U+(\Omega_\cF-[S,id])^*V)-(V+(\Omega_\cF-[S,id])^*V) \\= (U+(\Omega_\cF-[S,id])^*V)-(\Omega_\cF)^*V.
\end{multline}
The $\cF$-stability of the biset $\Omega_\cF$ implies that $(\Omega_\cF)^*V$ is $\cF$-stable. Since $X$ is also $\cF$-stable, the sum
\[X + (\Omega_\cF)^*V = U+(\Omega_\cF-[S,id])^*V\]
has to be $\cF$-stable as well. Hence \eqref{eqDifferenceOfFstableReps} expresses $X$ as a difference of $\cF$-stable representations, so $X\in \gen{R_\KK^+(\cF)}$ as we wanted.
\end{proof}

Using the biset functor structure of $R_\KK(-)$, we note the following special case of Proposition \ref{propIdempotentBisetTransfer} for future reference:
\begin{prop}
Let $\cF$ be a saturation fusion system on $S$. After $p$-localizing the representation rings, we have
\[R_\KK(\cF)_{(p)} = \omega_\cF \cdot R_\KK(S)_{(p)},\]
and $\omega_{\cF}$ sends each $\cF$-stable virtual representation to itself.
\end{prop}

The character ring $R_\KK(S)$ over $\KK$ embeds in the complex representation ring $R_\CC(S)$ by the map
\[R_\KK(S)\xto{\CC[S]\ox_{\KK[S]} -} R_\CC(S),\]
which is injective according to \cite{Serre}*{page 91}.
If we identify $R_\CC(S)$ with the ring of complex characters, the map sends each virtual representation $V\in R_\KK(S)$ to its character $\chi_V\colon S\to \KK\leq \CC$ defined as $\chi_V(s):=\text{Trace}(\rho_V(s))\in \KK$.

While the complex character for each $\KK[S]$-representation only takes values in $\KK$, it is not true that a complex representation with $\KK$-valued character has to come from a representation over $\KK$.

\begin{definition}
We define $\bar{R_\KK(S)}$ to be the subring of $R_\CC(S)$ consisting of complex characters that take all of their values in the subfield $\KK \leq \CC$.
\end{definition}

\begin{fact}\label{fact:SchurIndex}
The inclusion of rings $R_\KK(S)\leq \bar{R_\KK(S)}$ is of finite index. A proof of this fact can be found in \cite{Serre}*{Proposition 34}.
\end{fact}

\begin{definition}\label{defGaloisAction}
Let $\LL:= \KK(\zeta)\leq \CC$ be the extension of $\KK$ by some $(p^n)$th root of unity such that all irreducible complex characters for $S$ take their values in $\LL$. If all elements of $S$ have order at most $p^r$, then taking $\zeta$ as a $(p^r)$th root of unity will work.

Given any irreducible character $\chi$ and automorphism $\sigma\in \Gal(\LL / \KK)$, we define $\lc\sigma\chi \colon S \to \LL$ as
\[(\lc\sigma\chi)(s) := \sigma(\chi(s)) \text{ for $s\in S$},\]
which is another irreducible character of $S$. Taking the sum over all $\sigma\in \Gal(\LL / \KK)$ gives a character with values in the Galois fixed points, i.e. in $\KK$, so we define a transfer map $\tr\colon R_\CC(S) \to \bar{R_\KK(S)}$ by
\[\tr(\chi) := \sum_{\sigma\in \Gal(\LL / \KK)} \lc\sigma\chi.\]
Subsequently multiplying with index of $R_\KK(S)$ in $\bar{R_\KK(S)}$ gives a transfer map $R_\CC(S)\to R_\KK(S)$.
\end{definition}

Characters of $S$-representations are in particular class functions on $S$, i.e. functions defined on the conjugacy classes of $S$.
\begin{definition}
We denote the set of all complex valued class functions on $S$ by $\cl(S;\CC)$.
\end{definition}

\begin{remark}
The set of complex class functions $\cl(-;\CC)$ is a biset functor over all finite groups (see \cite{Bouc}*{Lemma 7.1.3}). Suppose $X$ is a $(G,H)$-biset and $\chi\in \cl(H;\CC)$ is a complex class function. The induced class function $X^*(\chi)$ is then given by
\[X^*(\chi)(g)=\frac1{\abs G} \sum_{\substack{x\in X,h\in H\\\text{s.t. }gx=xh}}\chi(h)\]
for $g\in G$. For complex characters $\chi\in R_\CC(H)$ the formula above coincides with the biset structure on representation rings (see \cite{Bouc}*{Lemma 7.1.3}).

Note that in the special case of restriction along a homomorphism $\ph\colon G\to H$, we just have
\[\res_\ph (\chi)(g) = \chi(\ph(g))\]
as we would expect.
\end{remark}

\begin{lemma}\label{lemFstableCharacters}
Let $\cF$ be a saturated fusion system on $S$, and let $\chi\in\cl(S;\CC)$ be any class function. Then $\chi$ is $\cF$-stable if and only if $\chi(s)=\chi(t)$ for all $s,t\in S$ that are conjugate in $\cF$ (meaning that $t=\ph(s)$ for some $\ph\in \cF$).
\end{lemma}

\begin{proof}
For every homomorphism $\ph\colon P\to S$, we have $\res_\ph(\chi)=\chi\circ \ph$. $\cF$-stability of $\chi$ is therefore the question whether $\chi\circ \ph = \chi|_P$ for all $P\leq S$ and $\ph\in\cF(P,S)$, which on elements becomes $\chi(\ph(s)) = \chi(s)$ for all $s\in P$.

We immediately conclude that $\chi$ is $\cF$-stable if and only if $\chi(\ph(s)) = \chi(s)$ for all $s\in P$, $P\leq S$ and $\ph\in \cF(P,S)$. By restricting each $\ph$ to the cyclic subgroup $\gen{s}\leq P$, it is enough to check that $\chi(\ph(s)) = \chi(s)$ for all $s\in S$ and $\ph\in \cF(\gen{s},S)$, i.e. that $\chi$ is constant on $\cF$-conjugacy classes of elements in $S$.
\end{proof}

\begin{prop}
Let $\LL$ be as in Definition \ref{defGaloisAction}, and let $\chi\in R_\CC(\cF)$ be any $\cF$-stable complex character. For every Galois automorphism $\sigma\in \Gal(\LL/\KK)$ the character $\lc\sigma\chi$ is again $\cF$-stable, and the transfer map $\tr\colon R_\CC(S)\to \bar{R_\KK(S)}$ restricts to a map $\tr\colon R_\CC(\cF)\to \bar{R_\KK(\cF)}$.
\end{prop}

\begin{proof}
According to Definition \ref{defGaloisAction} the character $\lc\sigma\chi$ is given by $\lc\sigma\chi(s) = \sigma(\chi(s))$ for $s\in S$. That $\chi$ is $\cF$-stable means, according to Lemma \ref{lemFstableCharacters}, that $\chi(\ph(s)) = \chi(s)$ for all $s\in S$ and $\ph\in\cF(\gen s,S)$. This clearly implies that
\[\lc\sigma\chi(\ph(s)) = \sigma(\chi(\ph(s)))=\sigma(\chi(s)) = \lc\sigma\chi(s)\]
for all $s\in S$ and $\ph\in\cF(\gen s,S)$, so $\lc\sigma\chi$ is $\cF$-stable as well.
Consequently, the transfer map applied to $\chi$ is
\[\tr(\chi) = \sum_{\sigma\in \Gal(\LL / \KK)} \lc\sigma\chi\]
which is a sum of $\cF$-stable characters and thus $\cF$-stable.
\end{proof}

\section{Borel-Smith functions for fusion systems}
\label{sect:BorelSmith}

Let $G$ be a finite group. A super class function defined on $G$ is a function $f$ from the set of all subgroups of $G$ to integers that is constant on $G$-conjugacy classes of subgroups. The set of super class functions for $G$, denoted by $C(G)$, is a ring with the usual addition and multiplication of integer valued functions. As an abelian group, $C(G)$ is a free abelian group with basis $\{ \e _H \mid H \in Cl(G) \}$, where $\e _ H (K)=1 $ if $K$ is conjugate to $H$, and zero otherwise. Here $Cl(G)$ denotes a set of representatives of conjugacy classes of subgroups of $G$.

We often identify $C(G)$ with the dual of the Burnside group $A^* (G):=\Hom (A(G), \ZZ )$, where a super class function $f$ is identified with the group homomorphism $A(G)\to \ZZ$ which takes  $G/H$ to $f(H)$ for every $H \leq G$. This identification
also gives a biset functor structure on $C(-)$ as the dual of the biset functor structure on the Burnside group functor $A(-)$. Given an $(K,H)$-biset $U$, the induced homomorphism $C(U)\colon C(H) \to C(K)$ is defined by setting
\[(U \cdot f ) (X)= f (U ^{op} \times _K X)\]
for every $f\in C(H)$ and $K$-set $X$. For example, if $K \leq H$ and $U=H$ as  a $(K,H)$ biset with the usual left and right multiplication maps, then $(U \cdot f) (K/L) =f(H/L)$ for every $L \leq K$. So, in this case $U\cdot f $ is the usual restriction of the super class function $f$ from $H$ to $K$, denoted by $\res ^H _K f$.

\begin{remark}
The biset structure on super class functions described above is the correct structure to use when considering dimension functions of representations as super class functions. However super class functions also play a role as the codomain for the so-called homomorphism of marks $\Phi\colon A(G)\to C(G)$ for the Burnside ring. $\Phi$ takes each finite $G$-set $X$ to the super class function $H\mapsto \abs{X^H}$ counting fixed points. As a word of caution we note that, with the biset structure on $C(-)$ given above, the mark homomorphism $\Phi$ is \emph{not} a natural transformation of biset functors. There is another structure of $C(-)$ as a biset functor, where $\Phi$ \emph{is} a natural transformation, but that would be the wrong one for the purposes of this paper.
\end{remark}

\begin{definition} Let $\cF$ be a fusion system on a $p$-group $S$. A \emph{super class function defined on $\cF$} is a function $f$ from the set of subgroups of $S$ to integers that is constant on $\cF$-conjugacy classes of subgroups.
\end{definition}

Let us denote the set of super class functions for the fusion system $\cF$ by $C'(\cF)$. As in the case of super class functions for groups, the set $C'(\cF)$ is also a ring with addition and multiplication of integer valued functions. Note that since $C(-)$ is a biset functor (with the biset structure described above), we have an alternative definition for $C(\cF)$ coming from Definition \ref{defFstable} as follows:
\[ C(\cF) := \{f\in C(S)\mid \res_\ph f = \res_P f \text{ for all $P\leq S$ and $\ph\in \cF(P,S)$}\}.
\]
Our first observation is that these two definitions coincide.

\begin{lemma} Let $\cF$ be a fusion system on a $p$-group $S$. Then $C'(\cF) =C(\cF)$ as subrings of $C(S)$.
\end{lemma}

\begin{proof} Let $f \in C'( \cF )$. Then for every $\varphi \in \cF (P, S)$ and $L \leq P$, we have
\[(\res _{\varphi} f ) (P/ L)=f( \ind _{\varphi} (P/L))=f(S/ \varphi (L) )=f(S/L)=(\res _P f )(P/L)\]
where $\ind _{\varphi} \colon A(P)\to A(S)$ is the homomorphism induced by the biset \[([P, \varphi]^S _P )^{op}= [\varphi(P), \varphi ^{-1}] ^P _S.\]
This calculation shows that $f \in C(\cF)$. The converse is also clear from the same calculation.
\end{proof}

The main purpose of this paper is to understand the dimension functions for representations in the context of fusion systems.
Now we introduce the dimension function of a representation.

Let $G$ be a finite group, and $\KK$ be a subfield of complex numbers $\CC$. As before let $R_\KK (G)$ denote the ring of $\KK$-linear representations of $G$  (see Definition \ref{defRepRing}). Recall that $R_\KK(G)$ is a free abelian group generated by isomorphism classes of irreducible representations of $G$ and the elements of $R_{\KK} (G)$ are virtual representations $V-W$. The \emph{dimension function} of a virtual representation $X=V-W$ is defined as the super class function $\Dim (X)$ defined by
\[\Dim (X)(H)= \dim _\KK (V^H)-\dim_\KK (W^H)\]
for every subgroup $H \leq G$. This gives a group homomorphism
\[\Dim \colon R_\KK (G) \to C(G)\]
which takes a virtual representation $X=V-W$ in $R_\KK (G)$ to its dimension function $\Dim (X) \in C(G)$.

\begin{lemma} The dimension homomorphism $\Dim \colon R_{\KK} (G) \to C(G)$ is a natural transformation of biset functors.
\end{lemma}

\begin{proof} Note that for a $\KK$-linear $G$-representation $V$, we have $\dim _{\KK} V^H =\dim _{\CC} (\CC \otimes _{\KK} V) ^H$, so we can assume $\KK =\CC$. Consider the  usual inner product in $R_{\CC} (G)$ defined by $\langle V , W \rangle =\dim _{\CC} \Hom _{\CC [G]} (V, W)$.
Since
\[\dim _{\CC} V^H =\langle \res ^G _H V , 1_H \rangle=\langle V, \ind _H ^G 1_H \rangle= \langle V, \CC[G/H]\rangle,\]
it is enough to show that for every $(K,H)$-biset $U$, the equality
\[\langle \CC [U] \otimes _{\CC [H]} V , \CC [K/L]\rangle=\langle V, \CC [U^{op} \times _K (K/L) ] \rangle\]
holds for every $H$-representation $V$ and subgroup $L \leq K$.
By Lemma \cite{Bouc}*{Lemma 2.3.26}, every $(K,H)$-biset is a composition of 5 types of bisets. For the induction biset the above formula follows from Frobenius reciprocity, for the other bisets the formula holds for obvious reasons.
\end{proof}

If $\cF$ is a fusion system on a $p$-group $S$, then recall (Proposition \ref{propVirtualFstableReps}) that every element  $X \in R_\KK (\cF) \subseteq R_\KK (S)$ can be written as
$X=V-W$ for two $\KK$-linear representations that are both $\cF$-stable. This implies that for every $X\in R_{\KK} (\cF)$, we have $\Dim (X) \in C (\cF)$ since $\Dim (X) (P)=\Dim (X) (Q)$ for every $\cF$-conjugate subgroups $P$ and $Q$ in $S$. Hence we conclude the following:

\begin{lemma} Let $\cF$ be a saturated fusion system on a $p$-group $S$. Then the dimension homomorphism $\Dim$ for $S$ induces a dimension homomorphism
\[ \Dim \colon R_\KK (\cF) \to C(\cF) \]
for the fusion system $\cF$.
\end{lemma}

For the rest of this section we assume $\KK=\RR$, the field of real numbers. The dimension function of a real representation of a finite group $G$ satisfies certain relations called Artin relations. These relations come from the fact that for a real represention $V$, the dimension of the fixed point set $V^G$ is determined by the dimensions of the fixed point sets $V^C$ for cyclic subgroups $C \leq G$ (see \cite{DovermannPetrie}*{Theorem 0.1}). In particular, the homomorphism $\Dim$ is not surjective in general. It is easy to see that in general $\Dim$ is not injective either. For example, when $G=C_p$, the cyclic group of order $p$ where $p\geq 5$, all two dimensional irreducible real representations of $G$ have the same dimension function ($f(1)=2$ and $f(C_p)=0$).

For a finite nilpotent group $G$, it is possible to explain the image of the homomorphism $\Dim \colon R_{\RR} (G) \to C(G)$ explicitly as the set of super class functions satisfying certain conditions. These conditions are called Borel-Smith conditions and the functions satisfying these conditions are called Borel-Smith functions (see \cite{BoucYalcin}*{Def. 3.1} or \cite{tomDieckTransformation}*{Def. 5.1}).

\begin{definition}
\label{def:BorelSmith}
Let $G$ be finite group. A function $f \in C(G)$ is called a {\it
Borel-Smith function} if it satisfies the following conditions:

\begin{enumerate}
\item\label{def:BorelSmith1} If $L \lhd H \leq G$, $H/L \cong \ZZ /p$, and $p$ is odd,
then $f(L)-f(H)$ is even.

 \item\label{def:BorelSmith2} If $L \lhd H \leq G$, $H/L \cong \ZZ /p \times \ZZ /p$, and $H_i/L$ are all the subgroups of order $p$ in $H/L$, then
     \[f(L)-f(H)=\sum
_{i=0} ^p (f(H_i)-f(H)).\]

\item\label{def:BorelSmith3} If $L \lhd H \lhd N \leq N_G(L)$ and $H/L \cong \ZZ /2$,
then $f(L)-f(H)$ is even if $N/L \cong \ZZ /4$, and $f(L)-f(H)$ is divisible by $4$ if $N/L$ is the quaternion group $Q_8$ of order $8$.
\end{enumerate}
\end{definition}

These conditions were introduced as conditions satisfied by dimension functions of $p$-group actions on mod-$p$ homology spheres and they play an important role understanding $p$-group actions on homotopy spheres (see \cite{DotzelHamrick}).

\begin{remark} The condition $(iii)$ is stated differently in \cite{tomDieckTransformation}*{Def. 5.1} but as it was explained in \cite{BoucYalcin}*{Remark 3.2} we can change this condition to hold only for $Q_8$ because every generalized quaternion group includes a $Q_8$ as a subgroup.
\end{remark}

The set of Borel-Smith functions is an additive subgroup
of $C(G)$ which we denote by $C_b(G)$. For $p$-groups, the assignment $P \to C_b(P)$ is a subfunctor of the biset functor $C$ (see \cite{Bouc}*{Theorem 12.8.7} or \cite{BoucYalcin}*{Proposition 3.7}). The biset functor $C_b$ plays an interesting role in understanding the Dade group of a $p$-group (see \cite{BoucYalcin}*{Theorem 1.2}). We now quote the following result for finite nilpotent groups.

\begin{theorem}[Theorem 5.4 on page 211 of \cite{tomDieckTransformation}]\label{thm:nilpotent} Let $G$ be a finite nilpotent group, and
let~$R_{\RR } (G) $ denote the real representation ring of $G$.
Then, the image of
\[\Dim \colon R_{\RR } (G) \to C(G)\]
is equal to the group of Borel-Smith functions $C_b(G)$.
\end{theorem}

If $\cF$ is a saturated fusion system on a $p$-group $S$, then we can define Borel-Smith functions for $\cF$ as super class functions that are both $\cF$-stable and satisfy the Borel-Smith conditions. Note that the group of Borel-Smith functions for $\cF$ is equal to the group
\[ C_b (\cF):= \{f\in C_b(S)\mid \res_\ph f = \res_P f \text{ for all $P\leq S$ and $\ph\in \cF(P,S)$}\}.
\]
From this we obtain the following surjectivity result.

\begin{prop}\label{p-localSurj} Let $\cF$ be a saturated fusion system on a $p$-group $S$. Then the dimension function in $\ZZ _{(p)}$-coefficients
\[\Dim \colon R_{\RR} (\cF ) _{(p)} \to C_b (\cF)_{(p)} \]
is surjective.
\end{prop}

\begin{proof} This follows from Proposition \ref{propShortExactBisetFunctors} since $\Dim\colon R_\RR(P)\to C_b(P)$ is surjective for all $P\leq S$ by Theorem \ref{thm:nilpotent}.
\end{proof}

In general the dimension homomorphism $\Dim\colon R_{\RR} (\cF) \to C_b(\cF)$ for fusion systems is not surjective.

\begin{example}
Let $G=\ZZ/p \rtimes (\ZZ/p)^{\times}$ denote the semidirect product where the unit group $(\ZZ /p)^{\times}$ acts on $\ZZ/p$ by multiplication.
Let $\cF$ denote the fusion system on $S=\ZZ/p$ induced by $G$. All $\cF$-stable real representations of $S$ are linear combinations of the trivial representation $1$ and the augmentation representation $I_S$. This implies that all super class functions in the image of $\Dim$ must satisfy $f(1)-f(S)\equiv 0\pmod{p-1}$. But $C_b(\cF)=C_b(S)$ is formed by super class functions which satisfy $f(1)-f(S)\equiv 0\pmod 2$. Hence the homomorphism $\Dim$ is not surjective.
\end{example}

\section{Bauer's surjectivity theorem}
\label{sect:Bauer}

To obtain a surjectivity theorem for dimension homomorphism in integer coefficients we need to add new conditions to the Borel-Smith conditions. S. Bauer \cite{Bauer}*{Theorem 1.3} proved that under an additional condition there is a surjectivity result similar to Theorem \ref{thm:nilpotent} for the dimension functions defined on prime power subgroups of any finite group. Bauer considers the following situation:

Let $G$ be a finite group and let $\cP$ denote the family of all subgroups of $G$ with prime power order. Let $\cD _{\cP}(G)$ denote the group of functions $f \colon \cP \to \ZZ$, constant on conjugacy classes, which satisfy the Borel-Smith conditions $(i)$-$(iii)$ on Sylow subgroups and also satisfy the following additional condition:

\begin{enumerate}\setcounter{enumi}{3}
\item\label{def:BorelSmith4} Let $p$ and $q$ be prime numbers, and let $L \lhd H \lhd M \leq N_G(L) $ be subgroups of $G$
such that $H/L\cong \ZZ /p $ and $H$ a $p$-group. Then $f(L)\equiv f(H) \pmod{q^{r-l}}$ if $M/H\cong \ZZ /q^r $ acts on $H/L$ with kernel of prime power order $q^l$.
\end{enumerate}

The Borel-Smith conditions \ref{def:BorelSmith1}-\ref{def:BorelSmith3} together with this last condition \ref{def:BorelSmith4} are all satisfied by the dimension function of an equivariant $\ZZ/\abs G$-homology sphere (see \cite{Bauer}*{Proposition 1.2}). Note that a topological space $X$ is called a $\ZZ/\abs G$-homology sphere if its homology in $\ZZ/\abs G$-coefficients is isomorphic to the homology of a sphere with the same coefficients.  Note that if $X$ is an equivariant $\ZZ/\abs G$-homology sphere, then for every $p$-subgroup $H$ of $G$, the fixed point set $X^H$ is a $\ZZ/p$-homology sphere of dimension $\underline{n}(H)$. The function $\underline{n} \colon \cP \to \ZZ$ is constant on $G$-conjugacy classes of subgroups in $\cP$, and we define the dimension function of $X$ to be the function $\Dim _{\cP} X \colon \cP\to \ZZ$ defined by $(\Dim _{\cP } X) (H)=\underline n(H)+1$ for all $H \in \cP$.

For a real representation $V$ of $G$, let $\Dim _{\cP} V\colon \cP \to \ZZ$ denote the function defined by $(\Dim _{\cP} V) (H)=\dim _{\RR} V^H$ on subgroups $H \in \cP$. Note that $\Dim _{\cP} V$ is equal to the dimension function of the unit sphere $X=S(V)$. Since $X=S(V)$ is a finite $G$-CW-complex that is a homology sphere,
the function  satisfies all the conditions \ref{def:BorelSmith1}-\ref{def:BorelSmith4} by \cite{Bauer}*{Proposition 1.2}. Hence we have $\Dim _{\cP} V \in D_{\cP} (G)$. Bauer proves that the image of $\Dim _{\cP}$ is exactly equal to $\cD _{\cP}(G)$.

\begin{theorem}[Bauer \cite{Bauer}*{Theorem 1.3}]\label{thm:BauerForGroups}
Let $G$ be a finite group, and let $\Dim _{\cP}$ be the dimension homomorphism defined above. Then $\Dim_{\cP} \colon R_{\RR} (G) \to  {\mathcal D}_{\mathcal P} (G)$ is surjective.
\end{theorem}

\begin{remark}\label{rem:Bauer} Note that Bauer's theorem gives an answer for Question \ref{ques:main} for virtual representations. To see this, let $S$ be a Sylow $p$-subgroup of a finite group $G$. Given a Borel-Smith function $f\in C_b (S)$ which respects fusion in $G$, then we can define a function $f'\in {\mathcal D}_{\mathcal P} (G)$ by taking $f'(Q)=f(1)$ for every subgroup of order $q^m$ with $q \neq p$. For $p$-subgroups $P \leq G$, we take $f'(P)=f(P^{g})$ where $P^g$ is a conjugate of $P$ which lies in $S$.
If $f$ satisfies the additional condition \ref{def:BorelSmith4} for subgroups of $S$, then by Theorem  \ref{thm:BauerForGroups} there is a virtual representation $V$ of $G$ such that $\Dim _{\cP} V=f'$. The restriction of $V$ to $S$ gives the desired real $S$-representation which respects fusion in $G$.
\end{remark}

To obtain a similar theorem for fusion systems, we need to introduce a version of Bauer's Artin condition for Borel-Smith functions of fusion systems.

\begin{definition}\label{def:FusionCondition} Let $\cF$ be a fusion system on a $p$-group $S$. We say a Borel-Smith function $f\in C_b(S)$ satisfies Bauer's Artin relation if it satisfies the following condition:
\begin{enumerate}
\renewcommand{\theenumi}{$(**)$}\renewcommand{\labelenumi}{\theenumi}
\item\label{def:ArtinCondition} Let $L \lhd H \leq S$ with $H/L\cong \ZZ/p$, and let $\ph \in \Aut _{\cF} (H)$ be an automorphism of $H$ with $\ph(L)=L$ and such that the induced automorphism of $H/L$ has order $m$. Then $f(L)\equiv f(H)\pmod m$.
\end{enumerate}
The group of Borel-Smith functions for $\cF$ which satisfy this extra condition is denoted by $C_{ba} (\cF)$. In the proof of Theorem \ref{thm:main} we only ever use Condition \ref{def:ArtinCondition} when $H$ is cyclic, so we could require \ref{def:ArtinCondition} only for cyclic subgroups and still get the same collection $C_{ba} (\cF)$.
\end{definition}

We first observe that $\cF$-stable representations satisfy this additional condition.

\begin{lemma} Let $V$ be an $\cF$-stable real $S$-representation. Then the dimension function $\Dim V$ satisfies Bauer's Artin condition \ref{def:ArtinCondition} given in Definition \ref{def:FusionCondition}.
\end{lemma}

\begin{proof} Let $V$ be an $\cF$-stable real $S$-representation and let $L \lhd H \leq S$ and $\ph \in \Aut _{\cF} (H)$ be as in \ref{def:ArtinCondition}. If $f=\Dim V$, then
$f(L)-f(H)=\langle \res^S _H V, I_{H/L} \rangle$ where $I_{H/L}=\RR[H/L]-\RR[H/H]$ is the augmentation module of the quotient group $H/L$, considered as an $H$-module via quotient map $H \to H/L$. Here the inner product is the inner product of the complexifications of the given real representations. Since $H/L\cong \ZZ/p$, there is a one-dimensional character
$\chi\colon H \to \CC ^{\times}$ with kernel $L$ such that
\[I_{H/L}=\sum _{i \in (\ZZ /p)^{\times} } \chi ^{i} \]
where $\chi ^{i} (h) =\chi(h)^i=\chi (h^i )$ for all $h\in H$. Hence we can write
\[f(L)-f(H) =\sum_{i \in (\ZZ /p)^{\times} } \langle \res^S _H V, \chi ^i \rangle.\]
Since $V$ is $\cF$-stable, we have $\varphi ^* \res ^S _H V=\res ^S _H V$ which gives that
\[\langle \res^S _H V, \chi ^i \rangle =\langle \res ^S _H V, (\varphi^{-1} ) ^* \chi ^i \rangle\]
for every $i \in (\ZZ /p )^{\times}$. Let $J\leq (\ZZ/p) ^{\times}$ denote the cyclic subgroup of order $m$ generated by the image of $\varphi$ in $(\ZZ/p)^{\times}=\Aut(\ZZ/p)$. For every $j \in J$ and $i\in (\ZZ /p)^{\times}$, we have $\langle \res^S _H V, \chi ^i \rangle =\langle \res ^S _H V, \chi ^{ij} \rangle$, hence $f(L)-f(H)$ is divisible by $m$.
\end{proof}

The main theorem of this section is the following result:

\begin{theorem}\label{thm:main} Let $\cF$ be a saturated fusion system on a $p$-group $S$. Then
\[\Dim \colon R_{\RR} (\cF ) \to C_{ba} (\cF) \]
is surjective.
\end{theorem}

\begin{proof} We will follow Bauer's argument given in the proof of \cite{Bauer}*{Theorem 1.3}. First note that a Borel-Smith function is uniquely determined by its values on cyclic subgroups of $S$. This is because if $P$ has a normal subgroup $N \leq G$ such that $P/N$ is isomorphic to $\ZZ/p\times \ZZ /p$, then by condition \ref{def:BorelSmith2} of the Borel-Smith conditions, the value of a Borel-Smith function $f$ at $P$ is determined by its  values on proper subgroups $Q <P$.  When $P$ is a noncyclic subgroup, the existence of a normal subgroup $N\leq P$ such that $P/N$ is isomorphic to $\ZZ/p\times \ZZ /p$ follows from the Burnside basis theorem, \cite{Suzuki}*{Theorem 1.16}.

Let $f \in C_{ba} (\cF)$. We will show that there is a virtual representation $x\in R_{\RR} (\cF)$ such that $\Dim (x) (H)=f(H)$ for all cyclic subgroups $H \leq S$. In this case we say $f$ is realized over the family $\cH _{cyc}$ of all cyclic subgroups in $S$.

Note that $f$ is realizable at the trivial subgroup $1 \leq S$ since we can take $f(1)$ many copies of the trivial representation, and then it will realize $f$ at $1$. Suppose that $f$ is realized over some nonempty family $\cH$ of cyclic subgroups of $S$. When we say a family of subgroups, we always mean that $\cH$ is closed under conjugation and taking subgroups: if $K \in \cH$ and $L$ is conjugate to a subgroup of $K$, then $L\in \cH$. Let $\cH'$ be an adjacent family of $\cH$, a family obtained from $\cH$ by adding the conjugacy class of a cyclic subgroup $H \leq S$. We will show that $f$ is realizable also over $\cH'$. By induction this will give us the realizability of $f$ over all cyclic subgroups.

Since $f$ is realizable over $\cH$, there is an element $x\in R_{\RR} (\cF)$ such that $\Dim (x) (J)=f(J)$ for every $J \in \cH$. By replacing $f$ with $f-\Dim (x)$, we assume that $f(J)=0$ for every $J\in \cH$. To prove that $f$ is realizable over the larger family $\cH'$, we will show that for every prime $q$, there is an integer $n_q$ coprime to $q$ such that $n_qf$ is realizable over $\cH'$ by some virtual representation $x_q \in R_{\RR } (\cF)$. This will be enough by the following argument: If $q_1, \dots, q_t$ are prime divisors of $n_p$, then $n_p, n_{q_1}, \dots, n_{q_t}$ have no common divisors, so we can find integers $m_0,\dots , m_t$
such that $m_0n_p+m_1 n_{q_1}+\dots + m_t n_{q_t}=1$. Using these integers, we obtain that $x=m_0 x_p+m_1x_{q_1}+\dots + m_t x_{q_t}$ realizes $f$ over the family $\cH'$.

If $q=p$, then by Proposition \ref{p-localSurj}, there is an integer $n_p$, coprime to $p$, such that $n_p f$ is realized by an element in $x_p \in R_{\RR} (\cF)$. So, assume now that $q$ is a fixed prime such that $q \neq p$. By \cite{Park}*{Theorem 1}, there is a finite group $\Gamma$ such that $S \leq \Gamma$ and the fusion system $\cF _S (\Gamma)$ induced by conjugations in $\Gamma$, is equal to the fusion system $\cF$. In particular, for every $p$-group $P \leq S$, we have $N_{\Gamma} (P)/C_{\Gamma} (P) \cong \Aut _{\cF} (P)$.

Let $H$ be a cyclic subgroup in $\cH' \backslash \cH$, and let $L\leq H$ be the maximal proper subgroup of $H$. To show that there is an integer $n_q$ such that $n_q f$ is realizable over $\cH'$, we can use the construction of a virtual representation $x_q$ given in \cite{Bauer}, specifically the second case in Bauer's proof of Theorem 1.3 that constructs the representation $V_q$, which we denote $x_q$ instead. Note that for this construction to work, we need to show Bauer's condition that $f(H)\equiv 0\pmod{q^{r-l}}$ if there is a $H\lhd K\leq \Gamma$ such that $K/H\cong \ZZ /q^r$ is acting on $H/L\cong \ZZ/p$ with kernel of order $q^l$. Suppose we have such a subgroup $K$.  Then $K/H$ induces a cyclic subgroup of $\Aut(H/L)\cong \Aut(\ZZ/p)$ of size $q^{r-l}$, and we can choose an element $k\in K$ that induces an automorphism of $H/L$ of order $q^{r-l}$. We have $\cF=\cF_S(\Gamma)$, so conjugation by $k\in K$ gives an automorphism $c_k\in \Aut_\cF(H)$ such that $c_k$ induces an automorphism of $H/L$ of order $q^{r-l}$. By Condition \ref{def:ArtinCondition} in Definition \ref{def:FusionCondition}, we then have $f(H)\equiv 0\pmod{q^{r-l}}$ as wanted. Hence we can apply Bauer's contruction of $x_q$ for primes $q\neq p$. This completes the proof.
\end{proof}

\begin{remark}  Note that in the above proof we could not apply Bauer's theorem directly to the group $\Gamma$ that realizes the fusion system on $S$. This is because $\Gamma$ often has a Sylow $p$-subgroup much bigger than $S$ and it is not clear if a given Borel-Smith function defined on $S$ can be extended to a Borel-Smith function defined on the Sylow $p$-subgroup of $\Gamma$.
\end{remark}

\begin{example} Let $\cF=\cF _S(G)$ with $G=A_4$ and $S=C_2\times C_2$, then the group $\Gamma$ constructed in \cite{Park} is isomorphic to the semidirect product $(S\times S \times S) \rtimes \Sigma _3$. The Sylow $2$-subgroup of $\Gamma$ is the group $T=(S\times S \times S) \rtimes C_2$ and $S$ is imbedded into $T$ by the map $s \to (s, \varphi (s), \varphi ^2 (s))$ where $\varphi \colon S \to S$ is the automorphism of $S$ induced by $G/S\cong C_3$ acting on $S$.
Note that if we extend a Borel-Smith function $f$ defined in $C_{b}(S)$ to a function defined on subgroups of $T$ by taking $f(H)=0$ for all $H$ which is not subconjugate to $S$, then such a super class function would not be a Borel-Smith function in general (for example, when $f(S)\neq 0$).
\end{example}

\section{Realizing monotone Borel-Smith functions}
\label{sect:monotone}

Let $\cF$ be a saturated fusion system on a $p$-group $S$. A class function $f \in C(\cF)$
is said to be \emph{monotone} if for every $K\leq H\leq S$, we have $f (K )\geq f(H) \geq 0$. The main purpose of this section is to prove the following theorem.

\begin{theorem}\label{thm:monotone} For every monotone Borel-Smith function $f \in C_b(\cF)$, there exists an integer $N \geq 1$ and an $\cF$-stable rational $S$-representation $V$ such that $\Dim V=Nf$.
\end{theorem}

For a $p$-group $S$, it is known that every monotone Borel-Smith function $f\in C_b(S)$ is realizable as the dimension function of a real $S$-representation. This is proved by  Dotzel-Hamrick in \cite{DotzelHamrick} (see also  \cite{tomDieckTransformation}*{Theorem 5.13}). Hence Theorem \ref{thm:monotone} holds for a trivial fusion system without multiplying with a positive integer.

\begin{question}\label{ques:FstableReal} Is every a monotone function $f\in C_{ba}(\cF)$ realizable as the dimension function of an $\cF$-stable real $S$-representation?
\end{question}

We leave this as an open problem. Note that we know from Theorem \ref{thm:main} that $f$ can be realized by a virtual real representation, and Theorem \ref{thm:monotone} says that a multiple $Nf$ is realized by an actual representation. We have so far found that Question \ref{ques:FstableReal} has positive answer for all saturated fusion systems on $C_p$ ($p$ prime) and $D_8$, and for the fusion system induced by $PGL_3(\FF_3)$ on $SD_{16}$.

Note also that Bauer's theorem (Theorem \ref{thm:BauerForGroups}) for general finite groups cannot be refined to realize a monotone function by an actual representation, as shown by the following example.

\begin{example}\label{example:MonotoneBSfunction} Let $G=S_3$ be the symmetric group on 3 elements. In this case $\cP$ is the set of subgroups $\{ 1, C_3, C_2^1,C_2 ^2, C_2 ^3 \}$ where $C_3$ is the subgroup generated by $(123)$ and $C_2^i$ is the subgroup generated by the transposition fixing $i$. Consider the class function $f\colon \cP \to \ZZ$ with values $f(1)=f(C_2^i)=2$, $f(C_3)=0$ for all $i$. The function $f$ satisfies all the Borel-Smith conditions and Bauer's Artin condition \ref{def:BorelSmith4}, hence $f\in \cD _{\cP} (G)$. By Bauer's theorem $f$ is realizable by a virtual representation. In fact, it is realized by the virtual representation $1_G-\sigma+\chi$ where $1_G$ is the 1-dimensional trivial representation, $\sigma$ is the 1-dimensional sign representation with kernel $C_3$, and $\chi$ is the unique 2-dimensional irreducible representation. The function $f$ is a monotone function, but there is no actual real $G$-representation $V$ such that
$\Dim V=f$. This can be easily seen by calculating dimension functions of these irreducible real representations. In fact, for any $N \geq 1$, the function $Nf$ is not realizable by an actual real $G$-representation. On the other hand, it is easy to see that the restrictions of $f$ to Sylow $p$-subgroups of $G$ for $p=2,3$ give monotone functions in $C_{ba} (\cF)$ and these functions are realized by $\cF$-stable real representations. So the function $f$ does not give a counterexample to Question \ref{ques:FstableReal}.
\end{example}

To prove Theorem \ref{thm:monotone}, we use the theorem by Dotzel-Hamrick \cite{DotzelHamrick} and some additional properties of rational representations. One of the key observations for rational representations that we use is the following.

\begin{lemma}\label{lem:detection} Let $G$ be a finite group and $V$ and $W$ be two rational representations of $G$. Then $V\cong W$ if and only if for every cyclic subgroup $C\leq G$, the equality $\dim V^C=\dim W^C$ holds.
\end{lemma}

\begin{proof} See Corollary on page 104 in \cite{Serre}.
\end{proof}

This implies, in particular, that the dimension function $\Dim \colon R_{\QQ} (S) \to C_b (S)$ is injective. Another important theorem on rational representations is the
 the Ritter-Segal theorem. We state here a version by Bouc (see \cite{Bouc}*{Section 9.2}). Note that if $G/N$ is a quotient group of $G$, then a $G/N$-representation $V$ can be considered a $G$-representation via the quotient map $G \to G/N$. In this case this $G$-representation is called the inflation of $V$ and it is denoted by $\inf ^G _{G/N} V$.

\begin{theorem}[Ritter-Segal] Let $S$ be a finite $p$-group. If $V$ is a non-trivial simple $\QQ S$-module, then there exist subgroups $Q \leq P$ of $S$, with $\abs{P:Q}=p$, such that
\[V \cong \ind ^S _P \mathrm{inf} ^P _{P/Q} I_{P/Q}\]
where $I_{P/Q}$ is the augmentation ideal of the group algebra $\QQ [P/Q]$.
\end{theorem}

For any field $\KK$, there is a linearization map $\mathrm{Lin} _\KK \colon A(G) \to R_\KK (G)$ from the Burnside ring $A(G)$ to the representation ring of $\KK G$-modules, which is defined as the homomorphism that takes a $G$-set $X$ to the permutation $\KK G$-module $\KK X$. It follows from the Ritter-Segal theorem that the linearization map
\[\mathrm{Lin} _{\QQ} \colon A(S)  \to R_{\QQ} (S)\]
is surjective when $S$ is a $p$-group. Note this is not true in general for finite groups (see \cite{BartelDokchitser}). For fusion systems we have $p$-local versions of these theorems.

\begin{prop}\label{pro:p-localrat} Let $\cF$ be a fusion system on a $p$-group $S$. Then the dimension homomorphism $\Dim \colon R _{\QQ} (\cF)_{(p)} \to C_b (\cF)_{(p)}$ is injective, and the linearization map $\mathrm{Lin} _{\QQ} \colon A(\cF)_{(p)} \to R_{\QQ} (\cF)_{(p)}$ is surjective.
\end{prop}

\begin{proof} This follows from Proposition \ref{propShortExactBisetFunctors} once we apply it to the injectivity and surjectivity results for $p$-groups stated above.
\end{proof}

Now we are ready to prove Theorem \ref{thm:monotone}.

\begin{proof}[Proof of Theorem \ref{thm:monotone}]
Let $f\in C_b (\cF) $ be a monotone Borel-Smith function. By Dotzel-Hamrick theorem, there is a real $S$-representation $V$ such that $\Dim V=f$. Let $\chi$ denote the character for $V$ and $\LL=\QQ (\chi)$ denote the field of character values of
$\chi$. The transfer
\[\tr(\chi) = \sum_{\sigma\in \Gal(\LL / \QQ)} \lc\sigma\chi\]
gives a rational valued character. Hence there is an integer $m$ such that $m \tr(\chi)$ is the character of a rational $S$-representation $W$. Note that $\Dim W= N f$ in $C_b (\cF)$, where $N= m \deg (\LL : \QQ )$. Let $\varphi \colon Q \to S$ be a morphism in $\cF$. Note that since $f$ is $\cF$-stable, we have
\[\Dim \res _{\varphi} W =N \res _{\varphi} f= N \res _Q f=\Dim \res _Q W.\]
The dimension function $\Dim$ is injective for rational representations by Lemma \ref{lem:detection}, we obtain that $\res _{\varphi} W=\res _Q W$ for every morphism $\varphi\colon Q \to S$. Therefore $W$ is an $\cF$-stable rational representation. This completes the proof.
\end{proof}

\begin{remark} In Theorem \ref{thm:monotone}, the integer $N\geq 1$ can be chosen independent from $f$. Note in the proof that the number $m$ only depends on $S$ (by Fact \ref{fact:SchurIndex}), and $\LL$ is contained in the extension $\LL '$ of $\QQ$ by roots of unity for the maximal element order in $S$, so if we take $N=m\deg (\LL' : \QQ )$, the conclusion of the theorem will hold for every $f \in C_b(\cF)$ using this particular $N$.

Note also that Theorem \ref{thm:monotone} will still hold if $f$ is an $\cF$-stable super class function which satisfies only the condition \ref{def:BorelSmith2} of the Borel-Smith conditions since the other conditions will be automatically satisfied by a multiple of $f$. This formulation is more useful for the applications for constructing group actions.
\end{remark}

\section{Applications to constructions of group actions}
\label{sect:applications}

In this section we discuss some applications of Theorem \ref{thm:monotone} to some problems related to finite group actions on homotopy spheres.

If a finite group $G$ acts freely on a sphere $S^n$, for some $n\geq 1$, then
by P.A. Smith theory $G$ can not include $\ZZ/p \times \ZZ /p$ as a subgroup for any prime $p$. The $p$-rank $\rk _p (G)$ of a finite group $G$ is defined to be the largest integer $s$ such that $(\ZZ /p)^s \leq G$. The rank of $G$, denoted by $\rk(G)$, is the maximum of the $p$-ranks $\rk_p (G)$ over all primes $p$ dividing the order of $G$. It is known, by a theorem of Swan \cite{Swan}, that a finite group $G$ acts freely and cellularly on a finite CW-complex $X$ homotopy equivalent to a sphere if and only if $\rk (G) \leq 1$. Recently Ian Hambleton and the second author were able to prove a similar theorem for rank two finite group actions on finite complexes homotopy equivalent to spheres (see \cite{HambletonYalcinRank1Prime}*{Theorem A}).

For rank two groups the classification involves the group $\Qd(p)$ which is defined as the semidirect product
\[\Qd(p) = (\ZZ/ p \times \ZZ / p)\rtimes SL_2(p)\]
with the obvious action of $SL_2(p)$ on $\ZZ / p \times \ZZ /p$. We say $\Qd(p)$ is $p'$-\emph{involved in $G$} if there exists a subgroup $K \leq G$, of order prime to $p$, such that $N_G(K)/K$ contains a subgroup isomorphic to $\Qd(p)$. A finite group $G$ is \emph{$\Qd(p)$-free} if it does not $p'$-involve $\Qd(p)$ for any odd prime $p$.

\begin{theorem}[Hambleton-Yal\c{c}{\i}n]\label{thm:ranktwo}
Let $G$ be a finite group of rank two. If $G$ admits a finite $G$-CW-complex $X\simeq S^n$ with rank one  isotropy then $G$ is $\Qd(p)$-free. Conversely, if $G$ is $\Qd(p)$-free, then there exists a finite $G$-CW-complex $X\simeq S^n$ with rank one prime power isotropy.
\end{theorem}

The proof of Theorem \ref{thm:ranktwo} uses a more technical gluing theorem (see Theorem \ref{thm:const} below). In this theorem the input is a collection of $G$-invariant family of Sylow representations. Let $G$ be a finite group. For every prime $p$ dividing the order of $G$, let $G_p$ be a fixed Sylow $p$-subgroup of $G$. A $G$-invariant family of representations is defined as follows.

\begin{definition}\label{def:Ginvariant}
Let $\{ V_p\} $ be a family of (complex) representations defined on Sylow $p$-subgroups $G_p$, over all primes $p$. We say the family $\{ V_p \}$ is \emph{$G$-invariant} if \begin{enumerate} \item  $V_p$ \emph{respects fusion in $G$}, i.e., the character $\chi _p$ of $V_p$ satisfies
$\chi_p (gxg^{-1})=\chi_p(x)$ whenever $gxg^{-1} \in G_p$ for some $g \in G$ and $x \in G_p$; and \item for
all $p$, $\dim V_p$ is equal to a fixed positive integer $n$. \end{enumerate}  \end{definition}

Note that if $\{ V_p\}$ is a $G$-invariant family of representations of $G$, then for each $p$, the representation $V_p$ is an $\cF$-stable representation for the fusion system $\cF=\cF _{G_p} (G)$ induced by $G$. The theorem \cite{HambletonYalcinRank1Prime}*{Theorem B} which allows us to glue such a family together is the following.

\begin{theorem}[Hambleton-Yal\c{c}{\i}n]\label{thm:const} Let $G$ be a finite group.   Suppose that $\{V_p\colon G_p \to U(n)\}$ is  a $G$-invariant family of Sylow representations. Then there exists a positive integer $k\geq 1$ and a finite $G$-CW-complex $X\simeq S^{2kn-1}$ with prime power isotropy,  such that  the $G_p$-CW-complex $\res ^G _{G_p} X$ is $p$-locally $G_p$-equivalent to $S(V_p ^{\oplus k} )$, for every prime $p \mid \abs G$,
\end{theorem}

The exact definition of $p$-local $G_p$-equivalence can be found in \cite{HambletonYalcinRank1Prime}*{Definition 3.6}. It says in particular that as $G_p$-spaces, the fixed point subspaces of $\res ^G _{G_p} X$ and $S(V_p ^{\oplus k } )$ have isomorphic $p$-local homology. From this we can conclude that for every $p$-subgroup $H \leq G$, the fixed point set $X^H$ has mod-$p$ homology of a sphere $S^{\underline{n} (H)}$, where $\underline{n}(H)=2k \dot \dim V_p ^H -1$.

Theorem \ref{thm:ranktwo} is proved by combining Theorem \ref{thm:const} with a theorem of Jackson \cite{Jackson}*{Theorem 47} which states that if $G$ is a rank two finite group which is $\Qd(p)$-free, then it has a $G$-invariant family of Sylow representations $\{ V_p\}$ such that for every elementary abelian $p$-subgroup $E$ with $\rk E=2$, we have $V_p^E = 0$.

The last condition is necessary to obtain an action on $X$ with the property that all isotropy subgroups are rank one subgroups. Note that all isotropy subgroups having rank $\leq 1$ is equivalent to saying that $X^E =\emptyset$ for every elementary abelian $p$-subgroups $E \leq G$ with $\rk E=2$. We can rephrase this condition in terms of the dimension function of a  homotopy $G$-sphere, which we describe now. Let $G$ be finite group and $\cP$ denote the collection of all subgroups of prime power order in $G$.

\begin{definition} Let $X$ be a finite (or finite dimensional) $G$-CW-complex such that $X$ is homotopy equivalent to a sphere. By Smith theory, for each $p$-group $H \leq G$, the fixed point subspace $X^H$ has mod-$p$ homology of a sphere $S^{\underline{n} (H)}$. We define $\Dim_{\cP} X \colon \cP \to \ZZ$ as the super class function with values $(\Dim _{\cP} X) (H)=\underline{n} (H)+1$ for every $p$-subgroup $H \leq G$, over all primes dividing the order of $G$.
\end{definition}

Using the results of this paper we can now prove the following theorem.

\begin{theorem}\label{thm:mainApp} Let $G$ be a finite group, and let $f\colon \cP \to \ZZ$ be a monotone Borel-Smith function. Then there is an integer $N\geq 1$ and a finite $G$-CW-complex $X \simeq S^n$, with prime power isotropy, such that $\Dim_{\cP} X=Nf$.
\end{theorem}

\begin{proof} For each prime $p$ dividing the order of $G$, let $\cF_p =\cF _{G_p} (G)$ denote the fusion system induced by $G$ on the fixed Sylow $p$-subgroup $G_p$. By Theorem \ref{thm:monotone} there is an integer $N_p$ and an $\cF_p$-stable rational $G_p$-representation $V_p$ such that $\Dim V_p=N_p f$. Taking $N'$ as the least common multiple of $N_p$'s, over all primes dividing the order of $G$, we see that $N'f$ is realized by a $G$-invariant family of Sylow representations $\{ W_p\}$ where $W_p$ is equal to the complexification of $(N'/N_p )$-copies of $V_p$. By Theorem \ref{thm:const}, there is an integer $k\geq 1$ and  a finite $G$-CW-complex $X$, with prime power isotropy, such that $\Dim_{\cP} X=2k N' f$. This completes the proof.
\end{proof}

\begin{example}While the monotone Borel-Smith function $f$ for $\Sigma_3$ in Example \ref{example:MonotoneBSfunction} has no multiple that is realized by a real $\Sigma_3$-representation, Theorem \ref{thm:mainApp} states that a multiple of $f$ is still realized by a homotopy $\Sigma_3$-sphere -- just not coming from a representation.
\end{example}

Theorem \ref{thm:mainApp} reduces the question of finding actions on homotopy spheres with restrictions on the rank of isotropy subgroups, to finding Borel-Smith functions that satisfy certain conditions. For example, using this theorem one can prove Theorem \ref{thm:ranktwo} directly now by showing that for every rank two finite group $G$ that is $\Qd(p)$-free, there is a monotone Borel-Smith function $f\colon \cP \to \ZZ$ such that $f(H)=0$ for every $H \leq G$ with $\rk H=2$.   Showing the existence of such a Borel-Smith function still involves quite a bit group theory, and the proof we could find runs through a lot of the same subcases as Jackson in \cite{Jackson}*{Theorem 47}, but with Borel-Smith functions instead of characters.

Theorem \ref{thm:mainApp} is also related to a question of Grodal and Smith on algebraic models for homotopy $G$-spheres. In \cite{GrodalAbs}*{after Thm 2}, it was asked if a given Borel-Smith function, defined on $p$-subgroups of a finite group $G$, can be realized as the dimension function of an algebraic homotopy $G$-sphere.  We describe the necessary terminology to state this problem.

Let $G$ be a finite group and $\cH _p$ denote the family of
all $p$-subgroups in $G$. The orbit category $\Gamma_G:={\mathrm Or} _p (G)$ is defined to be the category with objects $P \in \cH_p$, whose morphisms $\Hom _{\Gamma_G } (P, Q)$ are given by $G$-maps $G/P\to G/Q$. For a commutative ring $R$, we define an $R\Gamma_G$-module as a contravariant functor $M$ from $\Gamma_G$ to the category of  $R$-modules.

A chain complex $C_*$ of $R\Gamma _G$-modules is called \emph{perfect} if it is finite dimensional with $C_i$ a finitely generated projective $R\Gamma_G$-module for each $i$. A chain complex of $R\Gamma _G$-modules is said to be an $R$-homology $\underline{n}$-sphere if for each $H \in \cH _p$, the complex $C_* (H)$ is an $R$-homology sphere of dimension $\underline{n} (H)$. The dimension function of an $R$-homology $\underline{n}$-sphere $C_*$ over $R\Gamma_G$ is defined as a function $\Dim C_* \colon \cH_p \to \ZZ$ such that $(\Dim C_*)(H)=\underline{n} (H)+1$ for all $H \in \cH_p$.

It was mentioned in \cite{GrodalAbs}, and shown in \cite{HambletonYalcinHomotopyRep} that if $C_*$ is a perfect complex which is an $\FF_p$-homology $\underline{n}$-sphere, then the dimension function $\Dim C_*$ satisfies the Borel-Smith conditions (see \cite{HambletonYalcinHomotopyRep}*{Theorem C}). Grodal and Smith \cite{GrodalAbs}*{after Thm 2} suggests that the converse also holds.

\begin{question}[Grodal-Smith]\label{ques:GrodalSmith} Let $G$ be a finite group, and let $f$ be a monotone Borel-Smith function defined on $p$-subgroups of $G$. Is there then a perfect $\FF_p \Gamma _G$-complex $C_*$ which is an $\FF _p$-homology $\underline{n}$-sphere with dimension function $\Dim C_*=f$?
\end{question}

The motivation of Grodal-Smith for studying $\FF_p$-homology $\underline n$-spheres is that they are good algebraic models for $\FF_p$-complete homotopy $G$-spheres. The main claim of \cite{GrodalAbs} is that there is a one-to-one correspondence between $\FF_p$-complete homotopy $G$-spheres and $\FF _p$-homology spheres, with obvious low dimensional restrictions, and that $\FF _p $-homology spheres are determined by their dimension functions with an additional orientation \cite{GrodalAbs}*{Thms 2 and 3}. Based on \cite{GrodalAbs} Question \ref{ques:GrodalSmith} will then play a big part of determining all possible $\FF_p$-complete $G$-spheres.

We now show that Theorem \ref{thm:mainApp} can be used to give a partial answer to this question. Given a $G$-CW-complex $X$, associated to it there is chain complex of $R\Gamma _G$-modules $C_*$ defined by taking $C_* (H)= C_* (X^H ; R)$ for every $H \in \cH_p$ with associated induced maps. If $R=\FF_p$ and if $X$ has only prime power
isotropy then this chain complex is a chain complex of projective $R\Gamma _G$-modules. This follows from the fact that for a $q$-subgroup $Q$, with $q\neq p$, the permutation group $\FF _p [G/Q]$ is a projective $\FF _p G$-module. In addition if $X$ is a finite complex, then $C_*$ is a perfect $\FF_p \Gamma _G$-complex. So we can use Theorem \ref{thm:mainApp} to prove the following.

\begin{theorem}\label{thm:AlgMainApp} Let $G$ be a finite group, and $f$ be a Borel-Smith function defined on $p$-subgroups of $G$. Then there exists an integer $N \geq 1$ and a perfect $\FF_p\Gamma _G$-complex $C_*$ which is an $\FF_p$-homology $\underline{n}$-sphere with dimension function $\Dim C_*=Nf$.
\end{theorem}

\begin{proof} Let $f'$ be a Borel-Smith function defined on subgroups of $G$ with prime power order such that $f'(H)=f(H)$ for every $p$-subgroup $H \leq G$. Such a function can be found by taking $f'(K)=f(1)$ for all $q$-subgroups $K$ with $q\neq p$. Then by Theorem \ref{thm:mainApp}, there is an $N \geq 1$ such that $Nf$ is realizable as the dimension function of a finite $G$-CW-complex $X\simeq S^n$, with prime power isotropy. The chain complex $C_*=C_* (X ^{(-)}; R)$ is a perfect $\FF_p \Gamma _G$-complex which is an $R$-homology $\underline{n}$-sphere and we have $\Dim C_*=Nf$.
\end{proof}

\makeatletter
\def\eprint#1{\@eprint#1 }
\def\@eprint #1:#2 {%
    \ifthenelse{\equal{#1}{arXiv}}%
        {\href{http://front.math.ucdavis.edu/#2}{arXiv:#2}}%
        {\href{#1:#2}{#1:#2}}%
}
\makeatother

\end{document}